\documentclass[a4paper,twoside,11pt,reqno]{amsart}
\newcommand{\Ueberschrift}{\'Etale contractible varieties in positive characteristic}
\newcommand{\Kurztitel}{\'Etale contractible varieties in positive characteristic}
\usepackage{amsmath} \usepackage{amssymb} \usepackage{amsopn}
\usepackage{amsthm} \usepackage{amsfonts} \usepackage{mathrsfs}
\usepackage[headings]{fullpage}
\usepackage[bottom]{footmisc}
\usepackage[dvips]{graphicx} \usepackage[usenames]{color}
\usepackage[all]{xy}
\usepackage[OT2, T1]{fontenc}
\usepackage[utf8]{inputenc}
\usepackage{dsfont}
\usepackage[pdfpagelabels, pdftex]{hyperref}
\hypersetup{
  pdftitle={},
  pdfauthor={},
  pdfsubject={},
  pdfkeywords={},
  colorlinks=true,    
  linkcolor=red,     
  citecolor=blue,     
  filecolor=blue,      
  urlcolor=blue,       
  breaklinks=true,
  bookmarksopen=true,
  bookmarksnumbered=true,
  pdfpagemode=UseOutlines,
  plainpages=false}

\DeclareMathOperator{\rH}{H}

\DeclareMathOperator{\ra}{a}

\newcommand{\bA}{{\mathbb A}}

\newcommand{\bC}{{\mathbb C}}

\newcommand{\bF}{{\mathbb F}}
\newcommand{\bG}{{\mathbb G}}

\newcommand{\bP}{{\mathbb P}}
\newcommand{\bQ}{{\mathbb Q}}

\newcommand{\bZ}{{\mathbb Z}}

\newcommand{\cH}{{\mathscr H}}

\newcommand{\cU}{{\mathscr U}}

\newcommand{\cX}{{\mathscr X}}
\newcommand{\cY}{{\mathscr Y}}

\newcommand{\dO}{{\mathcal O}}

\DeclareSymbolFont{cyrletters}{OT2}{wncyr}{m}{n}
\DeclareMathSymbol{\Sha}{\mathalpha}{cyrletters}{"58}

\newcommand{\one}{\mathbf{1}}
\newcommand{\surj}{\twoheadrightarrow} 
\newcommand{\inj}{\hookrightarrow}

\DeclareMathOperator{\Hom}{Hom}

\DeclareMathOperator{\coker}{coker}

\DeclareMathOperator{\Spec}{Spec}

\DeclareMathOperator{\Pic}{Pic}

\DeclareMathOperator{\Br}{Br}

\DeclareMathOperator{\NS}{NS}

\DeclareMathOperator{\R}{R}

\newcommand{\et}{\text{\rm \'et}}

\newcommand{\ab}{{\rm ab}}

\newtheorem{thm}{Theorem}

\newtheorem{prop}[thm]{Proposition}
\newtheorem{lem}[thm]{Lemma}

\theoremstyle{definition}

\theoremstyle{remark}
\newtheorem{rmk}[thm]{Remark}

\newtheorem{ex}[thm]{Example}

\newenvironment{pro*}[1][Proof]{{\it{#1:}} }{}
\newenvironment{pro**}[1][]{{\it{#1}} }{\hfill $\square$}

\numberwithin{equation}{section}
\newcounter{absatzcounter}[section]
\usepackage{enum}

\newlist{enumer}{enumerate}{2}
\setlist[enumer]{label=(\roman*),align=left,labelindent=0pt,leftmargin=*,widest = (iii)}
\newlist{enumerar}{enumerate}{1}
\setlist[enumerar]{label=\arabic*.,align=left,labelindent=0pt,leftmargin=*,widest = 8.}
\newlist{enumera}{enumerate}{2}
\setlist[enumera]{label=(\arabic*),align=left,labelindent=0pt,leftmargin=*,widest = (8)}
\newlist{enumeral}{enumerate}{2}
\setlist[enumeral]{label=(\alph*),align=left,labelindent=0pt,leftmargin=*,widest = (m)}

\begin{document}

\hrule width\hsize
\hrule width\hsize

\vskip 0.9cm

\title[\Kurztitel]{\Ueberschrift} 
\author{Armin Holschbach}
\address{Armin Holschbach, Mathematisches  Institut, Universit\"at Heidelberg, Im Neuenheimer Feld 288, 69120 Heidelberg, Germany}
\email{holschbach@mathi.uni-heidelberg.de}
 
\author{Johannes Schmidt}
\address{Johannes Schmidt, Mathematisches  Institut, Universit\"at Heidelberg, Im Neuenheimer Feld 288, 69120 Heidelberg, Germany}
\email{jschmidt@mathi.uni-heidelberg.de}

\author{Jakob Stix}
\address{Jakob Stix, MATCH - Mathematisches  Institut, Universit\"at Heidelberg, Im Neuenheimer Feld 288, 69120 Heidelberg, Germany}
\email{stix@mathi.uni-heidelberg.de}

\thanks{Supported by DFG-Forschergruppe 1920 "Symmetrie, Geometrie und Arithmetik", Heidelberg--Darmstadt} 

\subjclass[2000]{14F35}
\keywords{\'etale homotopy theory, \'etale contractible}
\date{\today} 

\maketitle

\begin{quotation} 
\noindent \small {\bf Abstract} --- Unlike in characteristic $0$, there are no non-trivial smooth varieties
over an algebraically closed field $k$ of characteristic $p>0$  that are contractible in the sense of \'etale homotopy theory. 
\end{quotation}




\section{Introduction}

Homotopy theory is founded on the idea of contracting the interval, either as a space, or as an actual homotopy, i.e., a path in a space of maps. In algebraic geometry,  the affine line $\bA_k^1$ serves as an algebraic equivalent of the interval, at least in characteristic $0$, where $\bA^1_k$ is contractible.

Matters differ in characteristic $p>0$ where $\pi_1(\bA_k^1)$ is an infinite group: a group $G$ occurs as a finite quotient of $\pi_1(\bA^1_k)$ precisely if $G$ is a 
quasi-$p$-group 
due to Abhyankar's conjecture for the affine line as proven by Raynaud. 
This raises the question whether there is an \'etale contractible variety in positive characteristic. 

\begin{thm} \label{thm:maincontractible}
Let $k$ be an algebraically closed field of characteristic $p > 0$ and let $U/k$ be a smooth variety. Then $U$ is \'etale contractible, if and only if $U = \Spec(k)$ is the point.
\end{thm}

It turns out that our discussion in positive characteristic depends only on $\rH^1$ and $\rH^2$.
By the \'etale Hurewicz and Whitehead theorems (see \cite[\S4]{artinmazur}) we might therefore replace ``\'etale contractible'' with ``\'etale $2$-connected'' in Theorem~\ref{thm:maincontractible}.
Further, our proof covers more than just smooth varieties.
Here is the more precise result which proves Theorem~\ref{thm:maincontractible} because smooth varieties have big Cartier divisors.

\begin{thm} \label{thm:main2connected}
Let $k$ be an algebraically closed field of characteristic $p > 0$ and let $U/k$ be a normal variety such that 
\begin{enumer}
\item
the group $\rH^1_\et(U,\bF_p)$ vanishes, and
\item 
there is a prime number $\ell \not= p$ such that $\rH^2_\et(U,\mu_\ell) =  0$,
\end{enumer}
and one of the following: $U$ has a big Cartier divisor, or $\dim(U) \leq 2$.  Then $U$ has dimension $0$.
\end{thm}

In order to show the range of varieties to which Theorem~\ref{thm:main2connected} applies, 
we list in Proposition~\ref{prop:existbig} properties of varieties that imply the existence of big Cartier divisors, 
including quasi-projective varieties and locally $\bQ$-factorial (in particular smooth) varieties. 
The proof of Theorem~\ref{thm:main2connected} in the presence of a big Cartier divisor will be given in 
Section~\S\ref{sec:proof}.  The case of normal surfaces will be treated in Section~\S\ref{sec:surfaces}.

\smallskip

In the proof of Theorem~\ref{thm:main2connected} one would like to work with a compactification $U \subseteq X$ and the geometry of line bundles on $U$ versus $X$. For that strategy to work, we need a compactification that is locally factorial along $Y = X \setminus U$. Since in characteristic $p  > 0$ resolution of singularities is presently absent in dimension $\geq 4$, we resort to desingularisation by alterations due to de Jong. Unfortunately, the alteration typically destroys the \'etale contractibility assumption. 
We first deduce more \textit{coherent} properties from \'etale $2$-connectedness that transfer to the alteration.

The key difference with characteristic $0$ comes from Artin--Schreier theory relating $\rH^1_\et(U,\bF_p)$ to regular functions on $U$. 

\begin{rmk} 
Some further comments illustrate the situation in characteristic $0$ in contrast to Theorem~\ref{thm:maincontractible}.
\begin{enumera}
\item
There are contractible complex smooth surfaces other than $\bA^2_{\bC}$, the first such example is due to 
Ramanujam \cite[\S3]{ramanujam:contractible}, see also tom Dieck and Petrie \cite{tomdieckpetrie} for explicit equations.  All of them are affine and have rational smooth projective completions.
\item
Smooth varieties $U/\bC$ different from affine space $\bA^n_\bC$ but with $U(\bC)$ diffeomorphic to $\bC^n$ are known as exotic algebraic structures on $\bC^n$. These varieties are contractible and we recommend the Bourbaki talk on $\bA^n$ by Kraft \cite{kraft:bourbaki}, or the survey by Za{\u\i}denberg \cite{zaidenberg:exotic}. A remarkable non-affine (but quasi-affine) example $U$ was obtained by  Winkelmann \cite{winkelmann} as a quotient $U= \bA^5/\bG_{\ra}$ and more concretely as the complement in a smooth projective quadratic hypersurface in $\bP^5_\bC$ of the union of a hyperplane and a smooth surface.
\item
The notion of $\bA^1$-contractibility is a priori stronger than contractibility in the complex topology. In \cite{asokdoran} Asok and Doran construct for every $d \geq 6$ continuous families of pairwise non-isomorphic, non-affine smooth varieties of dimension $d$ that are even $\bA^1$-contractible.
\end{enumera}
\end{rmk}

\subsection*{Notation} 
We keep the following notation throughout the note: $k$ will be an algebraically closed field. 
By definition, a variety over $k$ is a separated scheme of finite type over $k$. We will denote the \'etale fundamental group by $\pi_1$ and its maximal abelian quotient by $\pi_1^\ab$. The sheaf $\mu_\ell$ for $\ell$ different from the characteristic denotes the (locally) constant sheaf of $\ell$-th roots of unity.


\subsection*{Acknowledgments} 
We would like to thank Alexander Schmidt and Malte Witte for comments on an earlier version of the paper.


\section{Big Cartier divisors} 
\label{sec:bigcartier}

\subsection{Existence of big divisors} 
\label{sec:big}

Recall that a Cartier divisor $D$ on a normal (but not necessarily proper) variety $U/k$ is big if the rational map associated to the linear system $|mD|$  is generically finite for $m \gg 0$. 

\begin{prop} \label{prop:existbig}
Let $k$ be an algebraically closed field and let $U/k$ be a normal variety such that one of the following holds:
\begin{enumeral}
\item \label{propitem:quasiprojective}
$U$ is quasi-projective.
\item \label{propitem:product}
$U$ is a product of varieties with big divisors.
\item \label{propitem:locallyfactorial}
$U$ is locally $\bQ$-factorial everywhere.
\end{enumeral}
Then $U$ has a big Cartier divisor.
\end{prop}
\begin{proof} 
Since any ample divisor is big, the conclusion holds if we assume \ref{propitem:quasiprojective}. In case \ref{propitem:product}, the sum of the pull backs of big Cartier divisors on the factors is again big. 

If \ref{propitem:locallyfactorial} holds, then we first choose a dense affine open $V \subseteq U$ and an effective big Cartier divisor $D$ on $V$ by \ref{propitem:quasiprojective}. 
Let $B = U \setminus V$ be the boundary, in fact a Weil divisor since $V$ is affine, and let $D'$ be the Zariski closure of $D$ as a Weil divisor on $U$. By assumption, 
$mD'$ and $mB$ are both effective Cartier divisors 
for $m \gg 0$ 
and there are sections $s_0,\ldots, s_d \in \rH^0(V,mD)$ such that the induced map $V \to \bP^d_k$ is generically finite. For $r \gg 0$ the sections $s_i$ extend to sections of 
\[
\rH^0(U,mD + mrB)
\]
so that $mD + mrB$ is the desired big Cartier divisor on $U$.
\end{proof}

\subsection{Geometry of varieties with vanishing \texorpdfstring{$\rH^1$ and $\rH^2$}{first and second cohomology}} \label{sec:vanishh1h2}

Let $\ell$ be a prime number different from the characteristic of $k$ and let  $U/k$ be a variety with $\rH^2_\et(U,\mu_\ell) = 0$.
It follows from the Kummer sequence in \'etale cohomology, that $\Pic(U)$ is an $\ell$-divisible abelian group.

\smallskip

The following 
crucially depends on $k$ being a field of positive characteristic.

\begin{prop} \label{prop:noh0}
Let $k$ be of characteristic $p> 0$. If $U/k$ is a connected reduced variety such that $\pi_1^\ab(U) \otimes \bF_p$ is finite, then $\rH^0(U,\dO_U) = k$. 
\end{prop}
\begin{proof}
We argue by contradiction. If  $f: U \to \bA^1_k$ is a dominant map, then the induced map 
\[
f_\ast : \pi^\ab_1(U) \otimes \bF_p \to \pi^\ab_1(\bA^1_k) \otimes \bF_p
\]
has image of finite index in the infinite group $\pi_1^\ab(\bA^1_k) \otimes \bF_p$, a contradiction.
\end{proof}

By the duality $\rH^1_\et(U,\bF_p) = \Hom(\pi^\ab_1(U), \bF_p)$, the vanishing of $\rH^1_\et(U,\bF_p)$ implies the assumption of Proposition~\ref{prop:noh0}.

\subsection{Using alterations} 
\label{sec:proof}

Section~\S\ref{sec:vanishh1h2} reduces the proof of Theorem~\ref{thm:main2connected} in the presence of a big Cartier divisor to the following proposition.

\begin{prop} \label{prop:main}
Let $k$ be an algebraically closed field and let $U/k$ be a connected normal variety with a big Cartier divisor and such that 
\begin{enumer}
\item \label{propitem:vanishHo}
$\rH^0(U,\dO_U) = k$, and
\item  \label{propitem:picelldiv}
there is a prime number $\ell$ such that $\Pic(U)$ is $\ell$-divisible.
\end{enumer}
Then $U$ has dimension $0$.
\end{prop}

\begin{proof} 
By \cite[Theorem~7.3]{Jon96}, there exists an alteration, i.e., a generically finite projective map $h: \tilde U \to U$ such that $\tilde U$ can be embedded into a connected smooth projective variety $\tilde X$. 

\smallskip

\textit{Step 1:}
The maximal open $V \subset U$, such that the restriction $\tilde{V}=h^{-1}(V) \to V$
is a finite map, has boundary $U \setminus V$ of codimension at least $2$, since $U$ is normal.

The $k$-algebra $\rH^0(\tilde{V},\dO_{\tilde{V}})$ is an integral domain inside the function field of $\tilde{V}$. The minimal polynomial for $s \in \rH^0(\tilde{V},\dO_{\tilde{V}})$ with respect to the function field of $V$ has coefficients that are regular functions on $V$ by normality and uniqueness of the minimal polynomial. Hence these coefficients are elements of 
$\rH^0(V,\dO_V) = \rH^0(U,\dO_U) = k$, and so 
\[
\rH^0(\tilde{V},\dO_{\tilde{V}}) = k.
\]

\smallskip

\textit{Step 2:}  
By  the theorem of the base \cite[Theorem~5.1]{kleiman:pic-sga6},  the N\'eron--Severi group 
\[
\NS(\tilde{X}) = \Pic(\tilde{X})/\Pic^0(\tilde{X})
\]
is a finitely generated abelian group. 
Since the restriction map $\Pic(\tilde{X}) \surj \Pic(\tilde{U})$ is surjective, the induced composite map 
\begin{equation} \label{eq:finiteimage}
h^\ast : \Pic(U) \to \coker\big(\Pic^0(\tilde{X}) \to \Pic(\tilde{U})\big)
\end{equation}
maps an $\ell$-divisible group to a finitely generated  abelian group, hence has finite image. 

\smallskip

\textit{Step 3:}  Let $D$ be a big Cartier divisor on $U$. Since $h: \tilde{U} \to U$ is generically finite, also the divisor $h^\ast{D}$ is a big Cartier divisor. 
Moreover, as in the proof of Proposition~\ref{prop:existbig}, there is a big divisor $\tilde{D}$ on $\tilde{X}$ that restricts to $h^\ast D$ on $\tilde{U}$. Upon replacing $D$ and $\tilde{D}$ by a positive multiple we may assume, by the finiteness of the image of the map \eqref{eq:finiteimage}, that $\tilde{D}$ is algebraically and thus numerically equivalent to a divisor $B$ on $\tilde{X}$ that is supported in $\tilde{X} \setminus \tilde{U}$.

Since bigness on projective varieties only depends on the numerical equivalence class, see \cite[Corollary~2.2.8]{Laz04}, the divisor $B$ is also big.  Restriction to $\tilde{V}$ yields
\[
\bigcup_{n \geq 0} \rH^0(\tilde X,\dO_{\tilde X}(nB)) \subseteq \rH^0(\tilde V, \dO_{\tilde V}) = k
\]
by Step 1 above. We conclude that $\dim(U) = \dim (\tilde{X}) = 0$ by bigness of $B$.
\end{proof}

\subsection{Complementing example} 

We illustrate the importance of the presence of a big divisor in Theorem~\ref{thm:main2connected} or Proposition~\ref{prop:main} by an example from toric geometry.

We first recall two facts about complete toric varieties that are standard analytically over $\bC$ and which have \'etale counterparts for toric varieties over arbitrary algebraically closed base fields, in particular of characteristic $p > 0$.

\begin{lem} \label{lem:toricsimplyconnected}
Let $k$ be an algebraically closed field. Any complete toric variety $X/k$ is \'etale simply connected: $\pi_1(X) = \one$.
\end{lem}
\begin{proof}
By toric resolution, see \cite[\S2.6]{Ful93},
there is a resolution of singularities $\tilde{X} \to X$ with a smooth projective toric variety $\tilde{X}$. 
Birational invariance of the \'etale fundamental group shows 
$\pi_1(\tilde{X}) = \pi_1(\bP^n_k) = \one$, and the surjection $\pi_1(\tilde{X}) \surj \pi_1(X)$ shows that $X$ is \'etale $1$-connected.
\end{proof}

\begin{lem} \label{lem:toricH2andPic}
Let $k$ be an algebraically closed field of characteristic $p$, and let $X/k$ be a complete toric variety. Then for all $\ell \not=p$ we have 
\[
\rH^2_\et(X,\bZ_\ell(1)) \simeq \Pic(X) \otimes \bZ_\ell.
\]
\end{lem}
\begin{proof}
In the context of toric varieties over $\bC$ and with respect to singular cohomology this is \cite[Corollary in 3.4]{Ful93}  
The $\ell$-adic case for toric varieties over an algebraically closed field $k$ of characteristic $\not= \ell$ follows with a parallel proof.
\end{proof}

\begin{ex}
Let $U=X$ be a complete normal non-projective toric variety $X$ of dimension $3$ with trivial Picard group. Such toric varieties have been constructed in \cite[Example 3.5]{Eik92}, or \cite[pp.~25--26, 65]{Ful93}. These sources construct $X$ over $\bC$ but the constructions work mutatis mutandis over any algebraically closed base field $k$. Then
\begin{enumer}
\item $\rH^1_\et(X,\bF_p)  = 0$  by Lemma~\ref{lem:toricsimplyconnected}, and
\item $\rH^2_\et(X,\bZ_\ell(1)) = 0$ for all $\ell \not=p$ by Lemma~\ref{lem:toricH2andPic}, and since there is non-trivial torsion in $\ell$-adic cohomology only for finitely many primes \cite{gabber:torsion}, we conclude that $\rH^2(X,\mu_\ell) = 0$ for almost all $\ell \not= p$. 
\end{enumer}
Therefore the assumptions of Theorem~\ref{thm:main2connected} hold with the exception of the presence of a big Cartier divisor. Nevertheless, these toric varieties are not \'etale contractible since $\rH^6_\et(X,\bZ_\ell(3)) = \bZ_\ell$. 
\end{ex}

\section{Normal surfaces}   
\label{sec:surfaces}
In this section, Proposition~\ref{prop:surfaces} completes the proof of Theorem~\ref{thm:main2connected} 
for surfaces.  Not every normal surface admits a big Cartier divisor, so something needs to be done. Examples of proper normal surfaces with trivial Picard group, in particular without big divisors, can be found in \cite{nagata:nonprojective} and \cite{schroeer:nonprojectivesurface}. 
However, on a hypothetical normal $2$-contractible surface a specialisation argument allows us to conclude the existence of a big  Cartier divisor in general.

\begin{prop} \label{prop:surfaces}
There is no normal connected surface $U/k$ over an algebraically closed field $k$ of characteristic $p>0$ such that 
\begin{enumer}
\item $\rH^1_\et(U,\bF_p)  = 0$, and
\item $\rH^2_\et(U,\mu_\ell) = 0$ for some prime number $\ell\not=p$.
\end{enumer}
\end{prop}
\begin{proof} 
We argue by contradiction and assume that $U$ is a surface as in the proposition. 
By Nagata's embedding theorem and resolution of singularities for surfaces, $U$ is a dense open in a normal proper surface $X/k$ with boundary $Y = X \setminus U$ being a normal crossing divisor. Hence, $X$ is smooth in a neighbourhood of $Y$. 

By limit arguments we may spread out over an integral scheme $S$ of finite type over $\bF_p$, i.e., there is  
a proper flat $f:\cX \to S$, a relative Cartier divisor $\cY$ in $\cX/S$ with normal crossing relative to $S$ and complement $\cU = \cX \setminus \cY$ such that the following holds:
\begin{enumeral}
\item all fibres are normal proper surfaces, and 
\item there is a point $\eta : \Spec(k) \to S$ over the generic point of $S$ such that the fibre over $\eta$ agrees with the original $\cX_\eta = X$ together with $\cU_\eta = U$ and $\cY_\eta = Y$, and 
\item the set of irreducible components of the fibres of $\cY$ forms a constant system, and each component of $\cY$ is a Cartier divisor, and 
\item  the higher  direct image $\R^2 {f|_\cU}_\ast \mu_\ell$ is locally constant and commutes with arbitrary base change by 
\cite[Finitude, Theorem 1.9]{sga4.5}. 
\end{enumeral}

Since the generic stalk $\big(\R^2 {f|_\cU}_\ast \mu_\ell\big)_\eta = \rH_\et^2(U,\mu_\ell) = 0$  vanishes, 
we conclude that for all geometric points $\bar s \in S$ we have $\rH_\et^2(\cU_{\bar s},\mu_\ell) = 0$, where $\cU_{\bar s}$ is the fibre of $\cU \to S$ in $\bar s$. As in the proof of Proposition~\ref{prop:main} this implies that for every Cartier divisor $D$ on $\cX_{\bar s}$ there is an $m \geq 1$ and a Cartier divisor $E$ on $\cX_{\bar s}$ supported in $\cY_{\bar s}$ such that 
$mD \equiv E$ are numerically equivalent. 

\smallskip

We apply this insight to a geometric fibre $\cX_{\bar t}$ above a closed point $t \in S$. 
Since by  a theorem of Artin \cite[Corollary~2.11]{artin:numericalcriteria}, all proper normal surfaces over the algebraic closure of a finite field  are projective, 
we conclude that there is a very ample Cartier divisor $H_{\bar t}$ on $\cX_{\bar t}$ with support contained in $\cY_{\bar t}$. 

Let $\cH \inj \cX$ be the relative Cartier divisor with support in $\cY$ that specialises to $H_{\bar t}$. By \cite[Theorem~4.7.1]{ega3}, the divisor $\cH$ is ample relative to $S$ in an open neighbourhood of $t \in S$. Consequently, the normal proper surface $X$ is projective, and in particular $U$ admits a big divisor.
The part of Theorem~\ref{thm:main2connected} proven in Section~\S\ref{sec:proof} leads to a contradiction.
\end{proof}

\begin{rmk}
It follows from the proof of  Proposition~\ref{prop:surfaces} that any proper non-projective normal surface $X$ with trivial Picard group, in particular the examples of \cite{nagata:nonprojective} and \cite{schroeer:nonprojectivesurface}, must have $\rH^2_\et(X,\mu_\ell) \not=0$ and a forteriori must contain non-trivial $\ell$-torsion classes in the cohomological Brauer group $\Br(X)$ for all $\ell$ different from the  characteristic. The existence of nontrivial torsion classes in $\Br(X)$ under the above assumptions was proven by different methods in \cite[proof of Theorem 4.1 on page 453]{schroeer:azumaya}.
\end{rmk}


\end{document}